\pdfoutput=1
\documentclass{amsart}

\usepackage{amsmath,amssymb,amsfonts,amscd}

\usepackage{enumerate}
\usepackage{ amssymb, latexsym, amsmath}
\usepackage{url}
\usepackage{hyperref}
\usepackage{graphicx}
\usepackage[all,cmtip]{xy}

\usepackage{xcolor}
\numberwithin{equation}{section}

\theoremstyle{plain}
\newtheorem{theorem}{Theorem}[section]

\newtheorem{proposition}[theorem]{Proposition}
\newtheorem{lemma}[theorem]{Lemma}

\newtheorem{corollary}[theorem]{Corollary}

\theoremstyle{definition}

\newtheorem{definition}[theorem]{Definition}

\newtheorem{example}[theorem]{Example}

\theoremstyle{remark}
\newtheorem{remark}[theorem]{Remark}

\makeatletter

\newcommand{\Rmnum}[1]{\expandafter\@slowromancap\romannumeral #1@}
\makeatother

\newcommand{\Z}{\mathbb{Z}}

\newcommand{\Q}{\mathbb{Q}}

\newcommand{\R}{\mathbb{R}}

\title{An HKR theorem for factorization homology}
\author{Hari Rau-Murthy}

\begin{document}

\maketitle

\begin{abstract}

We prove a Hochschild-Kostant-Rosenberg theorem (``the HKR theorem'') which computes the factorization homology of certain smooth commutative ring spectra.  In doing so, we fix and generalize a THH computation which was first conceived as the brainchild of McCarthy-Minasian.  A direct application of our revised HKR theorem is the higher THH of rational $KU$.  

\end{abstract}
\section{Introduction}
This paper concerns the following classical theorem due to Hochschild, Kostant and Rosenberg (HKR) \cite{HKR62}.  When $R \to A$ is a smooth map of commutative rings, we have an isomorphism
\begin{equation}
\label{hkranyring}
\mathrm{HH}_*(A/R) \to \Omega^*(A/R)
\end{equation}from the Hochschild homology relative to $R$, to the K\"{a}hler differential forms of $A$ relative to $R$.

The idea for generalizing (\ref{hkranyring}) to commutative ring spectra is due to McCarthy-Minasian in their main theorem of \cite{MM03}. In particular, they develop certain generalizations of a smooth map of classical rings to commutative rings. They state that
\begin{equation}\label{thhmainthm}\mathrm{THH}(A/R)\simeq  P_A( \Sigma \mathrm{TAQ}(A/R)) \end{equation}  
is an equivalence for suitable $A$.  Here, $P_A(-)$ is notation for the free commutative $A$-algebra on the $A$-module and $\mathrm{TAQ}(A/R)$ denotes the unaugmented Topological Andre Quillen homology of $A$ relative to $R$, developed by \cite{Bas99}, and discussed in Section \ref{TAQappendix}.

We give assumptions under which (\ref{thhmainthm}) is true.  While the precise assumptions on $R$ and $A$, including the smoothness assumption, and the proof of (\ref{thhmainthm}) have flaws in \cite{MM03}, McCarthy-Minasian provides the foundations we build on in this present paper.

The equivalence (\ref{thhmainthm}) is analogous to the classical HKR theorem (\ref{hkranyring}): informally, the left hand side of (\ref{thhmainthm}) is Hochschild homology and the right hand side is the graded commutative algebra over K\"{a}hler  1-forms concentrated in degree 1, which is the K\"{a}hler differential forms. 

% If you are using hyperref, it will say
%   Package hyperref Warning: Token not allowed in a PDF string 
% if you use math in a (sub)section title.  This is because it uses
% the title to produce bookmarks that can be shown in a side panel,
% and these can't use math.  You can use the \texorpdfstring command
% to give an alternate way to render these, possibly using unicode.

\begin{remark}
\label{classicalvsmodern}
 We remark that the equivalence (\ref{thhmainthm}) only implies the classical HKR theorem (\ref{hkranyring}) for $\Q$ algebras. A free commutative algebra in the classical setting is only rationally a free commutative algebra in spectra, owing to the homology of the symmetric group.  Additionally, the topological Andre Quillen homology, $\mathrm{TAQ}$,  only agrees rationally with algebraic Andre Quillen homology.  This problem is ubiquitous when switching from the setting of spectral algebraic geometry and derived/simplicial algebraic geometry as developed by
 \cite[Chapter 25]{Lur18}. 

\end{remark}

Among the errors in \cite{MM03}, the first two of which having been noted by Antieau, Toen and Vesozzi in \cite{AV20}\cite{TV08}, are the following:

\begin{enumerate}

\item MccCarthy-Minasian argue that (\ref{thhmainthm}) is true etale locally and hence it is true globally. They fail, however, to produce a global map which induces the etale local equivalences.
\item \cite{MM03} MccCarthy-Minasian assert that Theorem \ref{mainthm} implies the classical HKR theorem in all characteristics.  
\item \cite{MM03} MccCarthy-Minasian assume that two notions of the completion of a module with respect to an ideal, with one involving smash products over $A \otimes_R S^1$, and the other involving smash products over $A$, are the same.  
\end{enumerate}

Actually, the correct hypotheses of the main theorem of \cite{MM03} were stated in McCarthy-Minasian's subsequent paper \cite{MM04}.  Although \cite{MM04} refers to the flawed \cite{MM03} for details of a proof, the statements in \cite{MM04} are correct; moreover \cite{MM04} gives valuable insight, which we use, on how to correct the first error mentioned above of \cite{MM03}.

\subsection{Results}
We aim to fix and generalize (\ref{thhmainthm}).  The smoothness hypothesis we will need is as follows.  Classically a smooth map of rings is one which is locally etale over an affine space.    One of the many ways etale is interpreted in the context of commutative ring spectra is as follows.  For a map $A \to B$ of commutative rings, $\mathrm{TAQ}(B/A)$ is Topological Andre Quillen homology - a derived version of differential 1-forms.  We say that $A \to B$ is \textit{formally \textnormal{TAQ}-etale} when $\mathrm{TAQ}(B/A) \simeq  0$.  The corresponding notion of smoothness is formal TAQ smoothness.  For details, see Section \ref{smooth}.

\begin{theorem}
\label{mainthm}
Let $R \to A$ be a formally $\mathrm{TAQ}$-smooth map in the category of connective commutative rings and let $M$ be the realization of a pointed, connected simplicial set. 

Then the factorization homology is given by

\begin{equation}
\label{generalizedhkr}
\int_M (A/R):= A \otimes_R M \simeq  P_A(M \wedge \mathrm{TAQ}(A/R)).  
\end{equation}
\end{theorem}
We will call this theorem `the generalized HKR theorem.'
Here  $- \otimes_R M$ is the Loday functor of \cite{MSV97} discussed in Section \ref{lodayfunctor}.  In the case $M=S^1$ we recover Theorem \ref{thhmainthm}.

The Loday functor, $-\otimes_R S^n$, gives a different notion of smoothness which we call $S^n$-THH smooth (see Section \ref{smooth}).

%remind that its THH, is the Loday functor
%use align environment

In this setting, we also prove the following:
\begin{theorem}
\label{noconnectiveassumption}
Let $R \to A$ be a formally $S^n$-$\mathrm{THH}$-smooth map, and let $M$ be n-connected.  Then the conclusion \ref{generalizedhkr} of the generalized HKR theorem is satisfied.
\end{theorem}

Theorem \ref{noconnectiveassumption} depends on a certain etale descent lemma (Theorem \ref{lemmafornoconnectiveassumption}) which was independently proven in \cite[Proposition 2.11]{LR22} and proven in a special case in \cite[Theorem 7.5]{RSV22}.

Theorem \ref{mainthm} also includes as special cases $M = \mathbb{T}^n$, known as iterated THH (see \cite{CDD11}\cite{Sto20}\cite{BHLPRZ19}\cite{HKLRZ22}) and which accesses iterated K-theory important for the redshift conjecture as noted in \cite{HKLRZ22}\cite{CDD11}.  Also included is the case $M={S^n}$ which is a topological version of higher Hochschild homology (see \cite{Pir00}\cite{BHLPRZ19}).  These special cases were what initially motivated us to prove Theorem \ref{mainthm}.

In \cite{Sto20} and \cite{RSV22}, Stonek et al. prove that Theorem \ref{mainthm} is true for $R=\mathbb{S}$, $A=KU$, and any connected $M$ by direct computation. Since \cite{MM03} was released, several authors have proven etale descent results for THH \cite{Mat17}\cite{CM21}\cite{Rog08} and have proven etale descent results for higher THH \cite{Sto20}\cite{BHLPRZ19}\cite{RSV22}. %\cite{ABGHLM18} \cite{HKLRZ2022} \cite{CDD11}

\subsection{Organization of the paper}
We begin with a background section where we review the definition of $A \otimes_R M$  (Section \ref{lodayfunctor}) and facts about nonunital commutative algebras (Section \ref{NUCAdef}).  In Section \ref{notionsofetale}, we define $\mathrm{TAQ}$, the corresponding notion of etale, its generalization $M$-THH-etale, and other notions of etale.  In Section \ref{etaledescentsection}, we state and prove the etale descent lemmas which we need to prove our generalized HKR theorems. In Section \ref{polynomial}, we show that the generalized HKR theorem is satisfied for free commutative algebras.  In Section \ref{smooth} we develop a generalized notion of etale covering for commutative ring spectra.  In Section \ref{hiphooray}, we use this generalized notion of a covering to prove the generalized HKR Theorems \ref{mainthm} and \ref{noconnectiveassumption}.  
\subsection{Acknowledgements}
We gratefully acknowledge helpful conversations with Mark Behrens, Inbar Klang, Nikolai Konovolov, Andrew Blumberg, Ayelet Lindenstrauss, Connor Malin, Lorenzo Riva, Sihao Ma, Claudiu Raicu, John Rognes, and Nat Stapleton. We acknowledge support from the eCHT graduate research fellowship (NSF Grant DMS-1547292) for the academic year of 2022-2023 and support from the NSF Grant DMS-1547292 during the academic year of 2021-2022.  
\section{Background}
\subsection{The Loday construction}
\label{lodayfunctor}
Let $A$ be a commutative $R$-algebra, and let $M$ be the realization of a pointed simplicial set $M_\bullet$.  We will define the commutative $R$-algebra $A \otimes_R M$, known as the \textit{Loday construction}.  The foundations for this construction were first established in \cite{MSV97} and the terminology is due to \cite{HHLRZ18}. This construction is also quite standard and general and it can be found in any discussion on homotopy colimits.

$A \otimes_R M$ is formally defined as the $R$-algebra $\mathrm{colim}_M A$ where $M$ is viewed as an $\infty$-category and $A$ is taken to be the constant functor in $R$-algebras; we have not (yet) used the basepoint of $M$. By the definition of $A \otimes_R M$ as a colimit in an $\infty-$category, we have that up to equivalence, it only depends on the homotopy type of $M$, and that $A \otimes_R -$ commutes with colimits of spaces.  This will be used throughout the paper, and also gives us an explicit model for the Loday construction.

Explicitly, as shown below, $A \otimes_R M$ is the simplicial set which has a smash copy of $A$ for each simplex of $M_\bullet$.  The justification is as follows:  we have that 
\begin{align*}
A \otimes_R M\simeq A \otimes |M_\bullet|\\
\simeq A \otimes_R \mathrm{colim}_{\bullet \in \Delta}  M_\bullet\simeq \mathrm{colim}_{\bullet \in \Delta} A \otimes_R M_\bullet\\
\simeq| A^{\wedge_R M_\bullet}|.
\end{align*}

The explicit simplicial set structure on $A^{\wedge_R M_\bullet}$ comes from the observation that  $A \otimes_R (\mathrm{pt} \sqcup \mathrm{pt} \to \mathrm{pt})$ is identified with the fold or multiplication map $A \wedge_R A \to A$;  $A \otimes_R(\emptyset \to \mathrm{pt})$ is identified with the map from the initial object of $R-\mathrm{Alg}$ to $A$, or the unit map.  

This classical Loday construction is a special case of a more general Loday construction $c \otimes v$ where $C$ is an $\infty$-category enriched in a monoidal $\infty$-category $V$, and $c$ and $v$ are objects of $C$ and $V$ respectively.  We are interested only in the cases where $C=R-\mathrm{Alg}$, $V=\mathrm{Top}$, and $C=R-\mathrm{Mod}$, $V=\mathrm{Top}$.   The construction is as follows - consider the ($\infty$-)functor 
\[[c,-]: C_{\mathrm{un}} \to V \]
from the underlying $\infty$-category of $C$ to $V$.  When the unique left adjoint exists, it is denoted by $c \otimes -$; we will occasionally write $c \otimes^V_C v$ when it is necessary to emphasize $C$ and $V$, or $c \otimes_C v$ in the special case where $V=\mathrm{Top}$.  When $C$ is itself a presentable (unenriched)  $\infty$-category and $V=\mathrm{Top}$, the left adjoint exists because it is given by $v \mapsto \mathrm{colim}_Mc$; this expression exists because $C$ is cococomplete.  In particular, when $C=R-\mathrm{Alg}$, $V=\mathrm{Top}$, $c\otimes v$ is the classical Loday construction above, and using our above notation for rings and spaces, we have the following useful universal property:
\begin{equation}
\label{universalpropertytensor}
[A \otimes_R M, B]_{R-\mathrm{Alg}}\simeq  [M, [A,B]_{R-\mathrm{Alg}}]_{\mathrm{Top}}.
\end{equation}
This abstract framework gives rise to the following concrete application which we will need.  When $C=R-\mathrm{Mod}$ and $V=\mathrm{Top}$, $c \otimes^{\mathrm{Top}}_{C} v$ is given by $c \wedge (v_+)$.  This, for example, may be seen as a consequence of the fact that for $w \in \mathrm{ob}(\mathrm{Top}_+)$ we have that $c \otimes^{\mathrm{Top}_+}_{C} w \simeq c \wedge w$ and we also have the existence of the free functor from spaces to pointed spaces.  Regarding the free functor $P_R: R-\mathrm{Mod} \to R-\mathrm{Alg}$, we also have by the universal property of the general Loday constructions that
\begin{equation}
\label{foruseinpoly}
    P_R(X \otimes^{\mathrm{Top}}_{R-\mathrm{Mod}} M) \simeq P_R(X) \otimes_R^{\mathrm{Top}} M
\end{equation}
for any $R$-module $X$.

Lastly, the Loday construction is important to us as we have the equivalence $A \otimes_R S^1 \simeq  \mathrm{THH}(A/R)$ where the right hand side is defined via a cyclic bar construction \cite{MSV97}. This follows from the explicit model for $A \otimes_R -$ and the equivalence $S^1_\bullet \simeq  \Z/(\bullet+1)$. The n-fold \textit{iterated} THH is consequently \[(A \otimes_R S^1) \otimes_R S^1 \cdots \simeq  A \otimes_R (S^1 \times \cdots )\] where the last equality follows from $(A^{\wedge_R n})^{\wedge_R m}\simeq A^{\wedge_R nm}$.  

We have that $\mathrm{THH}(A/R)$ is an $A$-algebra via the map $A \to A \otimes S^1$ induced from $0 \to S^1$.  More generally, we would like to understand the $A$-algebra structure on $A \otimes M$ for pointed $M$ (and the tensor product $- \otimes^{\mathrm{Top}}_{R-\mathrm{Alg}}-)$), and we will need the following general fact about morphisms in under categories.

\begin{lemma}[{\cite[Lemma 5.5.5.12]{Lur09}}]
\label{undercategory}
    Let $C$ be an $\infty$ category, $c \in \mathrm{ob}(C)$, and $_{c\setminus}C$
    the undercategory for $c$.  Let $\underset{\sim}{d}:c \to d$ and $\underset{\sim}{e}:c \to e$ be objects in $_{c\setminus}C$.
    We then have a pullback square

\[  \xymatrix{  [\underset{\sim}{d},\underset{\sim}{e}]_{_{c\setminus}C} \ar[rr] \ar[dd] &&  \ar[dd] [d,e]_{C}  \\ \\
    \mathrm{pt} \ar[rr]^{\underset{\sim}{e}} && [c,e]_{C}}
   . \]
\end{lemma}
An immediate application of Lemma \ref{undercategory} and the universal property of the Loday construction (\ref{universalpropertytensor}) is the following change of base ring formula:
\begin{proposition}
    \label{ogbasering}
Let $R \to A \to B$ comprise of algebra maps.  Then we have the equivalence
\[(B \otimes_R M) \wedge_{A \otimes_R M} A \simeq  B \otimes_A M \]

\begin{proof}
For any $A$-algebra $C$, we have the following equivalences:
\[ [ (B \otimes_R M) \wedge_{A \otimes_R M} A, C]_{A-\mathrm{Alg}} \simeq [B \otimes_R M, C]_{A \otimes_R M - \mathrm{Alg}} \] \[ \simeq \mathrm{fib} ([B \otimes_R M,C]_{R-\mathrm{Alg}} \to [A \otimes_R M, C]_{R-\mathrm{Alg}}) \] \[\simeq \mathrm{fib}([M, [B,C]_{R-\mathrm{Alg}]_{\mathrm{Top}}} \to [M,[A,C]_{R-\mathrm{Alg}}]) \] \[\simeq[M, \mathrm{fib}([B,C]_{R-\mathrm{Alg}} \to [A,C]_{R-\mathrm{Alg}})]_{\mathrm{Top}}) 
 \] \[\simeq[M, [B,C]_{A-\mathrm{Alg}}]_{\mathrm{Top}} \simeq [B \otimes_A M,C]_{A-\mathrm{Alg}}\]
 By the Yoneda lemma we are done
\end{proof}
\end{proposition}
where the $A \otimes_R M$ module structure on $A$ is induced from the map $M \to \mathrm{pt}$. 

The dual statement to Lemma \ref{undercategory} is following lemma.
\begin{lemma}
\label{overcategory}
    Let $C$ be an $\infty$ category, $c \in \mathrm{ob}(C)$, and $C_{/c}$
    the overcategory for $c$.  Let $\widetilde{d}:d \to c$ and $\widetilde{e}: e \to c$ be objects in $C_{/c}$.
    We then have a pullback square

  \[  \xymatrix{  [\widetilde{d}, \widetilde{e}]_{C_{/c}} \ar[rr] \ar[dd] &&  \ar[dd] [d,e]_{C}  \\ \\
    \mathrm{pt} \ar[rr]^{\overset{\sim}{d}} && [d,c]_{C}}
   . \]
\end{lemma}
%change everything to curly C
\begin{lemma}
\label{undertensor}  There is a canonical equivalence
\[[c \otimes_C M, \underset{\sim}{d}]_{_{c \setminus}C}\simeq [M, [c,d]_C]_{{\mathrm{Top}_+}}.
\]

Here, we view $c \otimes_C M$ as an object of $_{c \setminus}C$ via the map $c \otimes_C \mathrm{pt} \to c \otimes_C M$.  

\begin{proof}
    By Lemma \ref{undercategory} and by (\ref{universalpropertytensor}), we have the following equivalences.
\begin{align*}
    [c \otimes_CM, \underset{\sim}{d}]_{_{c\setminus}C}\simeq \mathrm{fib}([M,[c,d]_C]_{\mathrm{Top}} \to [\mathrm{pt}, [c,d]_C]_{\mathrm{Top}})\\
    \simeq [M, [c,d]_C]_{\mathrm{Top}_+}.
\end{align*}
\end{proof}
\end{lemma}
%turn the equalities into simeq
\begin{lemma}

Dually, we have the canonical equivalence
\label{overtensor}
\[ \widetilde{c} \otimes_{C_{/c}} M \simeq c \otimes_C M.\] Here, the right hand side is interpreted as an object of $C_{/c}$ via the map $M \to \mathrm{pt}$, and $\widetilde{c}$ is notation for the object in $C_{/c}$ given by  $c \xrightarrow{\mathrm{Id}} c$.

\begin{proof}
By Lemma \ref{overcategory} and by (\ref{universalpropertytensor}), we have the following equivalences.
\begin{align*}
    [c \otimes_CM, \widetilde{d}]_{C_{/c}}\simeq \mathrm{fib}([M,[c,d]_C]_{\mathrm{Top}} \to [M, [c,c]_C]_{\mathrm{Top}})\\
    \simeq[M, [\widetilde{c},\widetilde{d}]_{C_{/c}}]_{\mathrm{Top}} \simeq [c \otimes_{C_{/c}} M, \widetilde{d}]_{\mathrm{Top}}.
\end{align*}
\end{proof}
\end{lemma}
 As an immediate corollary of Lemma \ref{undertensor}, we obtain a description for the $A$-algebra structure on $A \otimes_R M$.
\begin{proposition}
\label{Aalgebra}
The $A$-algebra structure on $A \otimes_R M$ may be understood via the following adjunction 
\[ [A \otimes_R M, B]_{A-\mathrm{Alg}}\simeq[M, [A,B]_{R-\mathrm{Alg}}]_{\mathrm{Top}_+}\]
where $[A,B]_{R-\mathrm{Alg}}$ has the base point given by the unique $A$-algebra map to $B$.
\end{proposition}

Lastly, for an $n$ manifold $M$, we have an equivalence 
\begin{equation}
    \label{factorizationtensor}
    \int_M A \to  A \otimes_{\mathbb{S}} M 
\end{equation}
from the $E_n$ factorization homology to the Loday construction relative to the sphere spectrum \cite[Proposition 5.1]{AF15}.

The proof is simple: since $A \otimes_{\mathbb{S}} -$ commutes with colimits of spaces, it is an excisive functor.  $E_n$ factorization homology is excisive for embedded framed $n$-disks by definition.  Constructing the map (\ref{factorizationtensor}) is the same as constructing the map for $M=\R^n$ via \[\int_{\R^n} A \simeq  A \simeq  A \otimes_{\mathbb{S}} \mathrm{pt} \simeq  A \otimes_{\mathbb{S}} \R^n. \] Since the above map is an equivalence, (\ref{factorizationtensor}) is an equivalence.

In particular we have an equivalence \[ \int_{S^n} A \xrightarrow{\simeq } A \otimes_{\mathbb{S}} S^n \] from $E_n$ factorization homology over $S^n$ to $n$-higher THH.

\subsection{Nonunital Commutative Algebras}
We give the basic properties of the $\infty$-category of nonunital commutative algebras, after \cite[Chapter 3]{Lur17} and \cite[Section 1]{Bas99}.  This will be used in the next section to define topological Andre Quillen homology, and to define the completions used throughout this paper.

 Let $A$ be a commutative ring. 

\begin{definition}
\label{NUCAdef}
The category $A$-NUCA of $A$-nonunital commutative algebras is defined as the $\infty$-category of symmetric monoidal functors of $\infty$-categories \[ (\mathrm{Fin}_{\mathrm{surj}}, \sqcup) \to \mathrm{Spectra}\]
from finite sets with surjections to spectra.

\end{definition}

To illuminate this definition, we compare to the definition of a commutative algebra.  The unit map of a commutative algebra is induced from $[0]=\emptyset \to [1] \in \mathrm{Fin}$; this is exactly the structure we remove to obtain a nonunital commutative algebra as in Definition \ref{NUCAdef}.

A related construction is the category of augmented $A$ algebras, defined by the slice category $A-\mathrm{Alg}/A$.  Explicitly, an object of this category is an $A$ algebra $U$ with a module equivalence $U \simeq A \vee I$ for some $A$ module $I$.  The commutative algebra structure on $A$ gives rise to an $A-\mathrm{NUCA}$ structure on $I$.  The functor which assigns to $A \vee I$, the $A-\mathrm{NUCA}$ $I$ is referred to as the \textit{augmentation ideal}.  This functor gives rise to the following equivalence \cite[Proposition 5.4.4.10]{Lur17} of $\infty$-categories

\begin{equation}
\label{babyadjunction}
\xymatrix{A-\mathrm{NUCA}  \ar@{.>}@/_1pc/[rr]|{A \vee -}  &&  \left. A-\mathrm{Alg} \middle/ A \right.   \ar@/_1pc/[ll]|{\text{aug. ideal}}}.
\end{equation} 
Thus both adjoints are simultaneously left and right adjoints. 

An important augmented $A$-algebra is $A \otimes_RM$, with the augmentation induced from the map $M \to \mathrm{pt}$. Denote the augmentation ideal by $I_M^A$.  For instance, $I_{S^0}^A=\mathrm{fib}(A \wedge_R A \to A)$ is the fiber of the multiplication map. We remark that this notation does not make explicit the base ring $R$; throughout this paper, the base ring will change and be understood from the context.

The augmented algebra structure on $A \otimes_R M$ is best understood via the following proposition.

\begin{proposition}
\label{augmentedtensor}
For any augmented $A$-algebra $U$, we have the following equivalence 
    \[ [ A \otimes_R M, U]_{A-\mathrm{Alg}/A} \simeq [M, [A,U]_{R-\mathrm{Alg}/A}]_{\mathrm{Top}_+}. \]
\begin{proof}
By applying Lemma \ref{overtensor} and then Lemma \ref{overtensor}, we obtain
\begin{align*}
[ A \otimes_R M, U]_{A-\mathrm{Alg}/A} \simeq [A \otimes_{R-alg/A} M, U]_{_{A \setminus}R-\mathrm{Alg}_{/A}}\\
\simeq [M, [A, U]_{R-\mathrm{Alg}_{/A}}]_{\mathrm{Top}_+}
\end{align*}
\end{proof}
\end{proposition}

\begin{definition} 
\label{quotients}

Let $N$ be an $A$-module, and let $I$ be a $A$-NUCA, with an associative map $I \wedge_A N \to N$.  Define $N/I^n$ via the cofiber sequence:
\[ I^{\wedge_An} \wedge_A N \to N \to N/I^n\]

\end{definition}
Examples of definition \ref{quotients} are following:

\begin{example}

\begin{enumerate}
\item the $A$-module \[ A \otimes_R M/(I_M^A)^n\] where $N=A \otimes_R M$, $I=I^A_M$
\item the $A$-module \[I_M^A/(I_M^A)^n\] where $N=I_M^A$, $I=I_M^A$
\item \[(I_M^A)^n/(I_M^A)^{n+1}\] where $N=(I_M^A)^{\wedge_A n}$ is given an $A$-module structure via any of the factors, and $I=I_M^A$.
It is also true that we have the equivalence \[(I_M^A)^n/(I_M^A)^{n+1} \simeq  \mathrm{fib}( I_M^A/(I_M^A)^{n+1} \to I_M^A/(I_M^A)^n).\]

\end{enumerate}
\end{example}

Definition \ref{quotients} also tells us how to make sense of completions:

\begin{definition}
In the setting of Definition \ref{quotients}, let \[ N^\wedge_I := \mathrm{lim}_n{ N/I^n}.\]
\end{definition}
 %$\textsc{e}$

\begin{proposition}
\label{connective}
 If $A$ is connective, $M$ is connected, and $N$ is bounded below, we then have the equivalence \[N \xrightarrow{\simeq } N^\wedge_{I_M^A}.\]
 
 \begin{proof}
 The map $A \otimes_R M \to A$ is an equivalence on $\pi_0$, since $M$ can be taken to be a simplicial set with a unique zero simplex.  Hence $I_M^A$ is connected, since $A$, $A \otimes_R M$ are connective.  Whence we obtain that $(I_M^A)^{\wedge_{A}n}$ is $n$-connected.
 
Thus the map \[ N \to N^\wedge_{I_M^A}\] is $n$-connected for all $n$ since $N$ is bounded below, and is therefore a weak equivalence.    
\end{proof}
\end{proposition}

\subsection{Topological Andre Quillen homology}

We define Topological Andre Quillen homology and discuss its basic properties.  This will be used to define brave new generalizations of the etale condition.  For a good complementary exposition, see Richter's survey paper \cite{Ric22} and for comparisons of different approaches, see \cite{RS20}.
\label{TAQappendix}

The topological Andre Quillen homology of $B$ relative to $A$ for a map $A \to B$ of commutative ring spectra, denoted $\mathrm{TAQ}(B/A)$, is designed to be a derived version of the classical Kahler 1-forms $\Omega^1(B/A)$ where $A \to B$ is a map of classical rings.  In the classical setting, $\Omega^1(B/A)\simeq I/I^2$ where $I$ is the augmentation ideal of the multipilcation map $B \otimes_A B \to B$.  This leads us the following definition for commutative ring spectra:

\begin{definition}
\label{firstdefnTAQ}
\[\mathrm{TAQ}(B/A):=I^B_{S^0}/(I^B_{S^0})^2.\]
\end{definition}

Regarding the motivation of Definition \ref{firstdefnTAQ}, we recall from the previous section that $I^B_{S^0}$ is the augmentation ideal of the map $B \wedge_A B \to B$.  

Another important property of $\mathrm{TAQ}(B/A)$ we will need is that it represents derivations:

\begin{proposition}
\label{defnTAQ}
For any $V \in B-\mathrm{Mod}$, we have the equivalence
\begin{equation}
\label{eqndefnTAQ}
[\mathrm{TAQ}(B/A), V]_{B-\mathrm{Mod}} \simeq  [ B,B \vee V]_{A-\mathrm{Alg}/B } .
\end{equation}
The right hand side of (\ref{eqndefnTAQ}) may be interpreted as derivations of $B$ over $A$ with values in $V$.
\begin{proof}
We have the following diagram of adjunctions
\[\xymatrix{
 && &&  B-\mathrm{Alg} \ar@/_1pc/[rr]|{\text{forget}} &&   \ar@{.>}@/_1pc/[ll]|{-\otimes_A B} A-\mathrm{Alg} \\ \\
  B-\mathrm{Mod} \ar@/_1pc/[rr]|{\mathrm{triv}} &&  B-\mathrm{NUCA} \ar@{.>}@/_1pc/[ll]|{I/I^2}    \ar@{.>}@/_1pc/[rr]|{B \vee -}  &&  \left. B-\mathrm{Alg} \middle/ B \right.  \ar@{->}[uu]   \ar@/_1pc/[ll]|{\text{aug. ideal}}  \ar@/_1pc/[rr]|{\text{forget}} &&   \ar@{->}[uu]       \ar@{.>}@/_1pc/[ll]|{-\otimes_A B}  \left. A-\mathrm{Alg} \middle/ B \right.
}\]

Here the middle adjunction is (\ref{babyadjunction}) and as noted there, both arrows are simultaneously left and right adjoints.  Observe that $\mathrm{TAQ}(B/A)$ is the composite of the arrows from $A-\mathrm{Alg}/B \to B-\mathrm{Mod}$ applied to the $A$-algebra $B$. We can then unravel the adjunctions to obtain that

\[ [\mathrm{TAQ}(B/A),V]_{B-\mathrm{Mod}}\simeq [(I/I^2) \circ (\text{aug. ideal}) \circ (- \otimes_A B) (B),V]_{B-\mathrm{Mod}} \] \[\simeq [B, \text{forget} \circ (B \vee -) \circ \mathrm{triv}(V)]_{A-\mathrm{Alg}/B}\simeq [B, B \vee V] .\]

\end{proof}
\end{proposition}

\subsection{Notions of etale:}
\label{notionsofetale}
In what follows, we define several brave new versions of etale.  We will show that these fit in the following dependency diagram:

\begin{equation}
\xymatrix{
  \text{classically etale} \ar@{=>}@/_0pc/[d]_-{(1)}  \\
  \text{Lurie-etale} \ar@{=>}@/_0pc/[r]_-{(2)} & \text{THH-etale descent}    \ar@{=>}@/_0pc/[r]_-{(3)}  &  \text{THH-etale}    \ar@{=>}@/_0pc/[r]_-{(4)} &   \ar@{==>}@/_1pc/[ll]_-{(5)}  \text{TAQ-etale}}
\label{dependencydiagram}
\end{equation}
Here, the implications (1),(2),(3),(4) always hold, and (5) holds under a connectivity hypothesis.  A classically etale map of classical rings is etale in every other sense by (\ref{dependencydiagram}), justifying the terminology.  A good summary of all of these notions of etale are given in Richter \cite[Sections 8.2,8.3]{Ric22}.

\begin{definition}

A ring map $A \to B$ is formally TAQ-etale iff $\mathrm{TAQ}(B/A) \simeq  0$.  For brevity we will simply say \textit{\textnormal{TAQ}-}etale.

\end{definition}

When we have a ring map $R \to A$, the condition $\mathrm{TAQ}(B/A) \simeq 0$ implies the equivalence
\begin{equation}
    \label{TAQetaledescent}
    \mathrm{TAQ}(A/R) \wedge_A B \cong \mathrm{TAQ}(B/R);
\end{equation}
this is true by the exact sequence/cofibration sequence of relative differential forms $\mathrm{TAQ}(A/R) \wedge_A B \to \mathrm{TAQ}(B/R) \to \mathrm{TAQ}(B/A)$, discussed in \cite[Proposition 4.2]{Bas99}.

We remark that when the simplicial cotangent complex vanishes for a map of classical rings $\textsc{a} \to \textsc{b}$, the map is typically called formally etale.  
\begin{definition}
\label{THHetale}
$A \to B$ is formally $S^1$-THH-etale if $B \simeq B\otimes_A \mathrm{pt} \to B\otimes_A S^1 :=\mathrm{THH}(B/A)$ is an equivalence.  For brevity, we will say \textit{\textnormal{THH-etale}}. 

\end{definition}

\begin{definition}
\label{Metale}
Given a pointed space $M$, we will say
$A \to B$ is $M$-THH-etale if 
\begin{equation} \label{augmentation} B\simeq  B\otimes_A \mathrm{pt} \to B\otimes_A M  \simeq  B \end{equation} is an equivalence.
\end{definition}
We remark that Bobkova et al in \cite{BHLPRZ19} call this notion $M$-etale.

\begin{remark}
\label{twooutofthree}
It is equivalent to ask that $B\otimes_A M \to B \otimes_A \mathrm{pt}$ is an equivalence, as done in \cite{MM03}.  Thus definition \ref{Metale} does not depend on the basepoint of $M$.
\begin{proof}[Proof of Remark \ref{twooutofthree}]
Consider the following diagram:
\[ \xymatrix{
B \simeq B \otimes_A \mathrm{pt} \ar@/^1pc/[rr]^-{(3)} \ar@/_0pc/[r]_-{(1)}  &  B \otimes_A M    \ar@/_0pc/[r]_-{(2)} &     B \otimes_A \mathrm{pt} }\]
The arrow (3) is always the identity map.  Thus by the two out of three property of equivalences, either of the other requirements implies the other.
\end{proof}
\end{remark}

\begin{lemma}
\label{tie}
For any $n$, $S^n$-$\mathrm{THH}$-etale implies formally $\mathrm{TAQ}$-etale.
\begin{proof}

By hypothesis, we have that the augmentation ideal is zero: $I_{S^n}^B \simeq  0$.  Thus $0 \simeq I_{S^n}^B/(I_{S^n}^B)^2 \simeq S^n \wedge \mathrm{TAQ}(B/A)$ ; here the second equality is true by Lemma \ref{linearity}, proven in Section \ref{etaledescentsection}. 

\end{proof}
\end{lemma}
In particular, we have implication (4) from our dependency diagram (\ref{dependencydiagram}).

The $S^0$-THH-etale condition is very strong.  It is not satisfied for Galois extensions of (classical) fields.  It is satisfied for the localization of a commutative (or even an $E_2$) ring at a homotopy group.

\begin{proposition}
\label{localizationexample}
Let $A$ be a commutative ring and let $x \in \pi_k A$.  Let $A[x^{-1}]:= \mathrm{colim}\, A \xrightarrow{x} \Sigma^{-k} A \xrightarrow{x}...$ where the arrows are given by left multiplication by $x$.  Then the map $A \to A[x^{-1}]$ is $S^0-\mathrm{THH}$-etale.

\begin{proof}

Note that  $A[x^{-1}]$ has an action of $A$ on the right.  Using the commutative structure, we see that it also has an action on the left, through the right action.  Hence we can write 
$A[x^{-1}] \wedge_A A[x^{-1}]=( \text{colim }A \xrightarrow{x} \Sigma^{-k} A \xrightarrow{x} ...) \wedge A[x^{-1}]\simeq  \text{colim }A[x^{-1}] \xrightarrow{x} \Sigma^{-k}A[x^{-1}] \xrightarrow{x} ...\simeq A[x^{-1}].$

\end{proof}
\end{proposition}
The last version of etale we consider is that of Lurie \cite[Chapter 7, Section 5]{Lur17}.
\begin{definition}
\label{lurieetale}

We will say that a map of commutative rings $A \to B$ is Lurie-etale if

\begin{enumerate}
\item The map on $\pi_0$ is etale.
\item The map  
\[ \pi_*(A) \otimes_{\pi_0(A)} \pi_0(B) \xrightarrow{\simeq } \pi_*(B) \]
is an isomorphism.
\end{enumerate}
\end{definition}

We note by definition that a classically etale map of rings $\textsc{a} \to \textsc{b}$ is Lurie-etale, giving us implication (1) in our dependency diagram \ref{dependencydiagram}.  

As we will see, $M$-THH-etale is a particular case of etale descent; thus to further study the basic properties of the higher $S^n$-THH-etale, $M$-THH-etale, Lurie-etale conditions and the implications (1),(2),(3),(5) in our dependency diagram, it is most expedient to study etale descent and its properties.

\section{Etale descent}
\label{etaledescentsection}

To prove Theorem \ref{mainthm} we will discuss a computation establishing it for the case $A =P_R X$, with $X$ an $R$-module.  We will then extend to the general case by proving that THH and its analogues satisfy etale descent.  As in the last section, $R \to A \to B$ will be maps of commutative rings and $M$ will be a pointed space.

\begin{definition}
We say that $ \otimes_R M$ satisfies etale descent along $A \to B$ whenever the map 
\begin{equation}
\label{etaledescent} (A \otimes_R M) \wedge_A B \xrightarrow{ \simeq } B \otimes_R M.
\end{equation}
is an equivalence.
\end{definition}

To see that we need an etale-type hypothesis, we observe the dependence of etale descent on the base ring $R$.
\begin{lemma}
\label{choiceofring}
Let $R \to S \to A \to B$ be maps of commutative rings, giving $A$ and $B$ the simultaneous structure of $R$-algebras and $S$-algebras.  Then if (\ref{etaledescent}) holds when the base ring is $R$, it holds when the base ring is $S$.

\begin{proof}
By Proposition \ref{ogbasering}, we have 
lo\[
(A \otimes_RM) \wedge_{S \otimes_R M} S  \simeq  A \otimes_S M.
\]
  Applying $\wedge_{S \otimes_R M} S$ to both sides of (\ref{etaledescent}), and using Proposition \ref{ogbasering}, we obtain that
 
 \[(A \otimes_S M) \wedge_A B \to B \otimes_S M\] is an equivalence.
\end{proof}
\end{lemma}

\begin{corollary}
If $-\otimes_R M$ satisfies etale descent along $A \to B$ (\ref{etaledescent}), then $A \to B$ is $M$-$\mathrm{THH}$ etale
\begin{proof}
    By Lemma \ref{choiceofring}, we may let the base ring $R=A$ to deduce that  
    \[ (A \otimes_A M) \wedge_A B \to B \otimes_A M\]
    is an equivalence.  We have that $A \otimes_A M \simeq A$ since $A^{\wedge_A n}\simeq A$.  Thus \[ B \to B \otimes_A M \] is an equivalence.
\end{proof}
\end{corollary}
This proves implication (3) of our dependency diagram (\ref{dependencydiagram}). Taken together with implication (4) proven in Lemma \ref{tie}, we have the following implication which explains the namesake of the term etale descent: when etale descent holds along $A \to B$ for $M=S^n$, we obtain that $\mathrm{TAQ}(B/A)\simeq 0$.

Next, we study the effect of the space $M$ in (\ref{etaledescent}).
\begin{lemma}
\label{buildup}
Suppose (\ref{etaledescent}) holds for spaces $X, Y,\text{ and } Z$.  Then (\ref{etaledescent}) is satisfied for $M\simeq X \cup_Y Z$.  

\begin{proof}
We have the following equivalence: \[ B \wedge_A (A \otimes_R(X \cup_Y Z))\simeq  B \wedge_A (A \otimes_R X \wedge_{A \otimes_R Y} A \otimes_R Z)\] 
\[\simeq   B \wedge_A (A \otimes_R X) \wedge_{B \wedge_A (A \otimes Y)}B \wedge_A(A \otimes_R Z)\simeq  B \otimes_R X \wedge_{B \otimes_R Y} B \otimes Z \] \[ \simeq B \otimes_R(X \cup_Y Z) \] 

\end{proof}

\end{lemma}
\begin{proposition}
\label{spheres}
Suppose (\ref{etaledescent}) is true for $M=S^n$, $n \geq 0$.  Then etale descent (\ref{etaledescent}) holds for any (n-1)-connected $M$.
\label{thhfactorization}
\begin{proof}

This proof is via Lemma \ref{buildup} and induction on cellular dimension. 
 Etale descent always holds for $M\simeq \mathrm{pt}$.  Etale descent holds for cell complexes concentrated in degree $k$, since $\vee_{\alpha} S^k$ is built using $X\simeq  Z\simeq  S^k,Y\simeq  \mathrm{pt}$.  Any $(k-1)$-connected $(n+1)$-skeletal space can be built as $X \cup_Y Z$ where $X\simeq  \mathrm{pt}$ and $Y,Z$ are $(k-1)$-connected $n$ skeletal spaces.  Assuming Proposition \ref{thhfactorization} is true for $n$-skeletal spaces which are $(k-1)$-connected, we obtain that its true for $(n+1)$-skeletal spaces which are $(k-1)$-connected.  By induction, etale descent holds for all $(k-1)$-connected spaces.  
\end{proof}
\end{proposition}

\begin{remark}
\label{literaturetome}
Proposition \ref{thhfactorization} in particular shows that $M-\mathrm{THH}$-etale descent holds for connected $M$ whenever etale descent holds for $\mathrm{THH}$. 
\end{remark}

We now prove the etale descent lemma needed for Theorem \ref{noconnectiveassumption}.  The following theorem was proven independently for $n=0$ in \cite[Theorem 7.5]{RSV22} and for arbitrary $n$ in \cite[Proposition 2.11]{LR22}.
\begin{theorem}
\label{lemmafornoconnectiveassumption}
Let $A \to B$ be $S^n$-$\mathrm{THH}$-etale.  Then etale descent (\ref{etaledescent}) holds for $M=S^{n+1}$, and hence whenever $M$ is $n$-connected by Proposition \ref{thhfactorization}.

\begin{proof}
 \[ B \wedge_A (A \otimes_R S^{n+1})\simeq  B \wedge_A A \wedge_{A \otimes_R S^{n}} A \] \[\simeq  B \wedge_{A \otimes_R S^n} A \simeq  B \wedge_{B \otimes_R S^n} B \otimes_R S^n \wedge_{A \otimes_R S^n} A \] \[\simeq  B \wedge_{B \otimes_R S^n} B \otimes_A S^n \simeq  B \wedge_{B \otimes_R S^n} B \] \[\simeq  B \otimes_R S^{n+1}\]
\end{proof}
\end{theorem}

In particular if $A \to B$ is $S^0$-THH-etale, then THH-etale descent holds along $A \to B$.  This implication was used in \cite{RSV22}.  In their paper, $S^0$-THH-etale extensions are called `solid'.

A byproduct of this theory on etale descent is nontrivial results for $M$-THH-etaleness.  In particular we have the following corollary.
\begin{corollary}
\label{forfree}
\begin{enumerate}
    \item If $A \to B$ is $X$,$Y$, and $Z$-$\mathrm{THH}$-etale, then $A \to B$ is $M$-$\mathrm{THH}$-etale where $M\simeq X \cup_Y Z$
    \item If $A \to B$ is $S^n$-THH-etale, then $A \to B$ is $M$-$\mathrm{THH}$-etale whenever $M$ is $(n-1)$ connected.
\end{enumerate}
\end{corollary}
We note that Corollary \ref{forfree} (2) is stronger than Theorem \ref{lemmafornoconnectiveassumption} in the setting of $M$-THH-etale.

Our next focus is on the remaining implications in our dependency diagram (\ref{dependencydiagram}).  If $A \to B$ is Lurie-etale, etale descent (\ref{etaledescent}) holds along $A \to B$ for $M=S^1$ and any choice of base ring by \cite{Mat17}; from this we obtain implication (2).  In addition, $M-\mathrm{THH}$-etale descent for any connected $M$ 
 holds along $A \to B$ by Proposition \ref{spheres}.  The last implication we have not proven yet is implication (5) which is exactly the content of the next theorem.  This etale descent theorem is also the main ingredient for Theorem \ref{mainthm}.  

\begin{theorem}
\label{connectiveetaledescent}
Let $R \to A \to B$ be maps of commutative rings with $A,B$ connective.  Suppose further that $\mathrm{TAQ}(B/A)\simeq  0$.  Then etale descent (\ref{etaledescent}) holds for $\mathrm{THH}(-/R)$ and equivalently for $-\otimes_R M$ with $M$ connected.
\end{theorem}

\begin{remark}
The version of Theorem \ref{connectiveetaledescent} proven in the setting of a Lurie-etale extension in \cite{Mat17} and \cite{CM21} is neither a special case nor a generalization of the connective, formal TAQ-etale setting here. Our proof follows the ideas of \cite{MM03}.
\end{remark}

\textit{Proof  of Theorem \ref{connectiveetaledescent}:}

We have a filtration on $A \otimes_R M$ by the powers of the augmentation ideal $I_M^A$ (Section \ref{NUCAdef} for a definition), giving us the following tower: 

\[ \xymatrix{ &  \vdots \ar[d]  \\
(I_M^A)^2/(I_M^A)^3  \ar@{^{(}->}[r]& A \otimes_R M/(I_M^A)^3 \ar[d] \\
I_M^A/(I_M^A)^2  \ar@{^{(}->}[r]& A \otimes_R M/(I_M^A)^2  \ar[d] \\
&A \otimes_R M/(I_M^A)  }\]

We now show that the comparison map \[(A \otimes_R M) \wedge_A B \to B \otimes_R M\] (\ref{etaledescent}) induces equivalent `associated graded's' or more precisely equivalences on the $n^{th}$ stages of the towers for $(A \otimes_R M) \wedge_A B$ and $B \otimes_R M$.  

For $n=1$, we note the following lemma:

\begin{lemma} \label{linearity}
\[ I_M^A/(I_M^A)^2 \simeq   M \wedge \mathrm{TAQ}(A/R). \]  
 In other words $I_M^A/(I_M^A)^2 $ is linear in $M$.
\begin{proof}
We have the equivalences
\[ [I_M^A/(I_M^A)^2 , V]_{A-\mathrm{Mod}} \simeq  [I_M^A, V]_{A-\mathrm{NUCA}} \simeq  [A \otimes_R M, A \vee V]_{A-\mathrm{Alg}/A}. \]

By Proposition \ref{augmentedtensor}, this is equal to $[M, [A, A\vee V]_{R-\mathrm{Alg}/A}]_{\mathrm{Top}_+}$.  On the other hand, we have the equivalences \[ [M \wedge \mathrm{TAQ}(A/R), V]_{A-\mathrm{Mod}} \simeq  [M,[\mathrm{TAQ}(A/R), V]_{A-\mathrm{Mod}}]_{\mathrm{Top}_+} \] \[\simeq  [M,[A, A \vee V]_{R-\mathrm{Alg}/A}]_{\mathrm{Top}_+}.\]

For the last equality, see Proposition \ref{defnTAQ}. Since this is true for all $V$, we are done by Yoneda's lemma.
\end{proof}
\end{lemma}

The comparison map (\ref{etaledescent}) induces an equivalence on the first stages of our towers: \[ I_M^A/(I_M^A)^2 \wedge_A B  \simeq  M \wedge \mathrm{TAQ}(A/R) \wedge_A B \] \[ \simeq  M \wedge \mathrm{TAQ}(B/R)  \simeq  I_M^B/(I_M^B)^2.\]

Here we use Lemma \ref{linearity}, and that $\mathrm{TAQ}$ satisfies etale descent, by (\ref{TAQetaledescent}).  For arbitrary $n$, we need the following lemma.

\begin{lemma}
\label{indecomposables}
Let $I$ be an $A$-$\mathrm{NUCA}$.  Then 
\[ I^n/I^{n+1}  \simeq    (I/I^2)^{\wedge_A n}_{h\Sigma_n} \]
\begin{proof}
This is the main result of \cite[Proposition 2.4]{Min03}.  The intuition is that the left hand side consists of products of $n$ elements which can not be written as a product of more elements.  The right hand side consists of products of $n$ indecomposables modulo the order in which they are written.

\end{proof}
\end{lemma}

\begin{lemma}
\label{honest}
Etale descent holds up to completion; more precisely, we have the following equivalence
\[ \left((A \otimes_R M) \wedge_A B \right)^\wedge_{I_M^A} \simeq  (B \otimes_R M)^\wedge_{I_M^B}.\]
\begin{proof}
    By Lemma \ref{indecomposables}, we have for arbitrary $n$ that
    \[ (I_M^A)^n/(I_M^A)^{n+1} \wedge_A B \simeq   \left((I_M^A)/(I_M^A)^{2}\right)^{\wedge_A n}_{h \Sigma_n}
 \wedge_A B \] \[\simeq  \left(M \wedge \mathrm{TAQ}(A/R) \wedge_A B\right)^{\wedge_B n}_{h \Sigma_n}
 \simeq   \left(M \wedge \mathrm{TAQ}(B/R) \right)^{\wedge_B n}_{h \Sigma_n} \] \[ \simeq  (I_M^B)^n/(I_M^B)^{n+1}.\]

Hence we have proven that the comparison map (\ref{etaledescent}) induces an equivalence on associated graded's and hence an equivalence on completion.

\end{proof}
\end{lemma}

Apply Proposition \ref{connective} to both sides of Lemma \ref{honest} to obtain (\ref{etaledescent}).  This completes the proof of Theorem \ref{connectiveetaledescent}. \qed

\section{Proof of Theorem 0.1 for polynomial algebras}

\label{polynomial}

We will prove \textit{Theorem 0.1} when $A=P_RX$ where $X$, as before, is an (arbitrary) $R$-module, and $M$ is a pointed space.\\

\begin{lemma}
    The functor $P_R(-)$ takes the tensor in the category $R-\mathrm{Mod}$ with an unpointed space, to the tensor in the category $R-\mathrm{Alg}$, with the same unpointed space.  That is 
    \[ P_R(X \wedge M_+)\simeq P_R(X) \otimes_R M.\]
\begin{proof}
This is a direct consequence of (\ref{foruseinpoly}).
\end{proof}
\end{lemma}

Thus we have that 
\begin{align*}
P_R(X) \otimes_R M\simeq P_R(X \wedge M_+)\\
\simeq P_R(X \wedge M \vee X)\simeq P_R(X) \wedge_R P_R(X \wedge M)\\
\simeq P_{P_R X}(P_RX \wedge_R X \wedge_R M). 
\end{align*}

Next we write down the right hand side using $\mathrm{TAQ}$ (see also \cite[Proposition 1.6]{BGR08}). 
\begin{proposition}
\label{polynomialTAQ}
We have the equivalence
\[\mathrm{TAQ}(P_R(X)/R) \simeq  P_R(X) \wedge_R X.\]

Intuitively, this is true because the left hand side is reminiscent of differential forms.  On the RHS, $P_R X$ reminds us of the coefficients of the differential forms, and $X$ reminds us of the span of the $dx_1... dx_n$. 

To prove this proposition, we will need the following lemma.

\begin{lemma}
\label{forpolynomial}
Let $A$ be an augmented $R$ algebra.  We then have an adjunction 
\[ \xymatrix{R-\mathrm{Alg}/A \ar@{.>}@/_1pc/[rr]|{I_R} &&  R-\mathrm{NUCA} \ar@/_1pc/[ll]|{- \vee A}}.\]

Here $I_R$ is the functor $R-\mathrm{Alg}/A \to R-\mathrm{Alg}/R \to R-\mathrm{NUCA}$.

\begin{proof}
Note that we assert $I_R$ is \textit{left}-adjoint to extension by $A$.  First we use the equivalence of categories \[ [I_R(B),J]_{R-\mathrm{NUCA}} \simeq  [B, R \vee J ]_{R-\mathrm{Alg}/R}\] from NUCA's to augmented algebras in \ref{babyadjunction}.  If we take this together with the equivalence \[[B, R \vee J]_{\mathrm{Alg}_R/R} \simeq  [B,A \vee J]_{\mathrm{Alg}_R/A} \]
we obtain the lemma.
\end{proof}
\end{lemma}

\begin{proof}[Proof of Proposition \ref{polynomialTAQ}]
We have that TAQ controls square zero extensions, as witnessed in the following equation  
\[ [\mathrm{TAQ}(A/R), V]_{A-\mathrm{Mod}} \simeq  [ A,A \vee V]_{R-\mathrm{Alg}/A } \] 
and discussed in (\ref{defnTAQ}).  Let $A=P_R X$.  The algebra $A$ is an augmented $R$ algebra and hence by Lemma \ref{forpolynomial} we have the equivalences \[[A,A \vee Z(V)]_{R-\mathrm{Alg}/A} \simeq  [I_R(A), Z(V)]_{R-\mathrm{NUCA}}\] \[ \simeq  [I_R(A)/I_R(A)^2,V]_{R-\mathrm{Mod}}. \]

In our case, this can be related back to $A-\mathrm{Mod}$ via the equivalences
\[  \simeq  [ X,V]_{R-\mathrm{Mod}} \simeq  [ X \wedge_R P_RX,V]_{P_RX-\mathrm{Mod}}\] and this establishes Proposition \ref{polynomialTAQ}, by Yoneda's lemma.

\end{proof}
\end{proposition}
Thus we have that $P_{P_RX}( P_RX \wedge_R X \wedge M)\simeq P_{P_RX}(M \wedge \mathrm{TAQ})$ as desired.  This proves Theorem \ref{mainthm} in this special case.

\section{Notions of coverings and smoothness}
\label{smooth}

We define the hypotheses used in Theorems \ref{mainthm} and \ref{noconnectiveassumption}.  These involve brave new notions of smooth; these notion of smoothness will depend on certain generalized notions of etale coverings that are engineered to rectify the main error of McCarthy and Minasian's first paper \cite{MM03} noted in our introduction.  As these notions of coverings may at first feel unnatural, we will also show that under mild hypotheses, these generalized notions of coverings are all equivalent to the classical notion of an etale covering, when working in the underived classical setting.

In this section, $Y$ and $Z$ will denote $A$-modules.  As in the rest of the paper, $\textsc{a}$ will denote a classical ring, and $\textsc{y}$ will denote a classical $\textsc{a}$-module.  Generalized etale will mean any of the generalized notions of etale from Section \ref{notionsofetale} including respectively TAQ, THH, M-THH, or Lurie etale.  Similarly for generalized etale cover, and generalized smooth, which will be defined presently.

\begin{definition}
\label{convenient}
Say that a collection of generalized etale commutative ring maps $\{A \to A_\alpha \}_\alpha$ is a generalized etale global covering if whenever an $A$ module map splits $A_\alpha$-locally, there is a splitting globally.

\end{definition}

Regarding Definition \ref{convenient}, we will use the following standard notation. 
 Given an $A$-module $Y$, we will let $Y_\alpha$ denote $Y \wedge_A A_\alpha$.  Given an $A$-module map $f: Y \to Z$, we will let $f_\alpha$ denote the induced map $f_\alpha: Y_\alpha \to Z_\alpha$.  In this notation, Definition \ref{convenient} can be formulated as follows:  a collection of generalized etale maps $\{A \to A_\alpha \}_\alpha$ is a generalized etale global covering if whenever $f: Y \to Z$ has $s_\alpha: Z_\alpha \to Y_\alpha$ such that $f_\alpha \circ s_\alpha \simeq \mathrm{Id}_{Z_\alpha}$, there is a map $s: Z \to Y$ such that $f \circ s \simeq \mathrm{Id}_Z$.  We note that Definition \ref{convenient} is the one given in McCarthy-Minasian's second paper \cite[Definition 6.9]{MM04}.

\begin{definition}
\label{formalsmooth}
\begin{enumerate}
\item Say that a map $R \to A$ is generalized smooth if there is a generalized etale global covering $\{A_\alpha\}_\alpha$ of  $A$ such that for each $A_\alpha$, there an $R$-module $X$ and a factorization

\[ R \to P_RX \xrightarrow{\phi} A_{\alpha}\]

with $\phi$ generalized etale. 

In other words $A$ is locally etale over $R$-affines. 

\item Given such maps, we say that $R \to A$ is generalized smooth in the category of connective commutative rings if $R,A, A_\alpha$ are all connective and for each $A_\alpha$, the corresponding $P_RX$ is also connective.
\end{enumerate}

\end{definition}

In order to compare this notion of a covering with more classical notions of a covering, we introduce another generalized notion of cover.  This is covering used in McCarthy-Minasian's first paper \cite{MM03}.
\begin{definition}
\label{faithfuldefinition}
Say that a collection of generalized etale commutative ring maps $\{A \to A_\alpha \}_\alpha$ is a faithful generalized etale covering if whenever a map $Y \to Z$  induces an equivalence  $A_\alpha$-locally, it is itself an equivalence.
\end{definition}

Definitions \ref{convenient}, \ref{formalsmooth}, \ref{faithfuldefinition} all have analogues for classical ring maps; we obtain these analogues by defining `etale local' using underived tensor products.  The maps $\textsc{a} \to \textsc{a}_\alpha$ will always be required to be classically etale.

Recall that $\textsc{r} \to \textsc{a}$ is classically smooth if it satisfies the conditions of Definition \ref{formalsmooth} where in place of a global covering, we use a classical etale covering: 

\begin{definition}
A classical etale covering is a collection of etale maps ${\textsc{a} \to \textsc{a}_\alpha}$ inducing a surjection $\cup \text{Spec }(\textsc{a}_\alpha) \to  \text{Spec } \textsc{a}$.
\end{definition}

\begin{proposition}
\label{reallynice}
In this classical setting the following is true.
\begin{enumerate}
    \item A faithful covering is a classical etale covering and vice versa.
    \item A global covering is always a faithful covering.
    \item  A faithful covering satisfies the requirements of a global covering when $\textsc{z}$(of Definition \ref{faithfuldefinition}) is finitely presented.
\end{enumerate}   

\end{proposition}
We remind the reader that all of the maps $\{\textsc{a} \to \textsc{a}_\alpha \}$ are required to be etale in all the notions of covering in Proposition \ref{reallynice}.  In particular, the map $\textsc{a} \to \underset{\alpha}{\Pi} \textsc{a}_\alpha$ will be flat. 

\noindent \textit{Proof of Proposition \ref{reallynice} (1)}:

We have that the flat map $\textsc{a} \to \Pi \textsc{a}_\alpha$ is faithfully flat iff any map between $A$-modules $f: \textsc{y} \to \textsc{z}$ is an isomorphism whenever $f \otimes_\textsc{a} \Pi \textsc{a}_\alpha$ is an isomorphism. This happens iff $\{\textsc{a} \to \textsc{a}_\alpha \}$ is a faithful covering.  We also have that the flat map $\textsc{a} \to \Pi \textsc{a}_\alpha$ is faithfully flat iff the induced map on Spec is surjective. This happens iff $\{\textsc{a} \to \textsc{a}_\alpha \}$ is an etale covering.\qed
\\
\noindent \textit{Proof of Proposition \ref{reallynice} (2)}:

Suppose we are given a map $f: \textsc{y} \to \textsc{z}$ such that $f_\alpha$ is an isomorphism for each $\alpha$.  From the condition of a global cover, we have that the existence of the $f_\alpha$'s imply that $f$ is split surjective and split injective, and hence is an isomorphism.  
\qed
\begin{remark}
\label{globalfaithful}
    The proof of Proposition \ref{reallynice} (2) also implies that global coverings are faithful coverings in the setting of commutative ring spectra.
\end{remark}

\noindent \textit{Proof of Proposition \ref{reallynice} (3)}:

We prove the following:  Let $\textsc{z}$ be finitely presented. Then any map $\textsc{y} \to \textsc{z}$ is split surjective iff it is locally split surjective.

Suppose $f: \textsc{y} \to \textsc{z}$ is locally split surjective.  This implies that $f$ is locally surjective, and since surjectivity can be checked on stalks, this implies $f$ is surjective.  Thus, let $\textsc{k}$ fit in the exact sequence \[0 \to \textsc{k} \to \textsc{y} \to \textsc{z} \to 0.\]  We will show that the corresponding element $\xi \in \mathrm{Ext}^1(\textsc{z}, \textsc{k})$ is 0.  Since $\textsc{z}$ is finitely presented, by definition it has a resolution by projectives \[\textsc{p}_1 \to \textsc{p}_0 \to \textsc{z}\] with $\textsc{p}_1, \textsc{p}_0$ finitely generated. The map \[\mathrm{Hom}(\textsc{p}_i,\textsc{k}) \otimes_\textsc{a} \textsc{a}_\alpha \to \mathrm{Hom}_{\textsc{a}_\alpha}(\textsc{p}_i \otimes_\textsc{a} \textsc{a}_\alpha, \textsc{k} \otimes_\textsc{a} \textsc{a}_\alpha)\] is an isomorphism for $i \simeq  0,1$ since the map is an isomorphism when $\textsc{p}_i$ is replaced by a finitely generated free module.  After taking cohomology, we deduce that the map \[\mathrm{Ext}^1(\textsc{z}, \textsc{k}) \otimes_{\textsc{a}} \textsc{a}_\alpha \to \mathrm{Ext}^1(\textsc{z} \otimes_{\textsc{a}} \textsc{a}_\alpha, \textsc{k} \otimes_{\textsc{a}} \textsc{a}_\alpha ) \] is an injection.  By hypothesis, the image of $\xi$ is zero in the right hand side, and hence in the left hand side.  Consider the submodule $ \langle \xi \rangle  \subset \mathrm{Ext}^1(\textsc{z},\textsc{k})$.  We have that the map $\langle \xi \rangle \to 0$ is an isomorphism $\textsc{a}_\alpha$-locally.  Hence the module $\langle \xi \rangle$ vanishes. 

\qed

\section{Proof of Theorems \ref{mainthm} and \ref{noconnectiveassumption}}
\label{hiphooray}
Consider $R, A, A_\alpha, P_R X$ as in Definitions \ref{convenient} and \ref{formalsmooth}.  We will show that if etale descent (\ref{etaledescent}) holds along each $A \to A_{\alpha}$ and $P_R X \to A_\alpha$, and that $\{A_\alpha\}_\alpha$ form a global covering, that equivalence (\ref{generalizedhkr}) holds. The hypotheses for Theorem \ref{mainthm} and Theorem \ref{noconnectiveassumption} guarantee that these conditions hold by Section \ref{etaledescentsection}.  By Section \ref{polynomial}, we have the equivalence
\[P_R X \otimes_R M \simeq  P_{P_R X} (M \wedge \mathrm{TAQ}(P_R X/R)).\]
Since etale descent holds, we have the equivalence
\[ \left(A_\alpha \wedge_{P_R X} P_R X \otimes_R M\right) \simeq   \left(A_\alpha \otimes_R M\right)\]
whence we obtain the equivalences 
\begin{equation} \label{firstpart}
\begin{split}
P_{A_\alpha}(M \wedge \mathrm{TAQ}(A_\alpha/R)) \simeq   A_\alpha \wedge_{P_R X} P_{P_RX}(M \wedge \mathrm{TAQ}(P_RX/R))   \\
\simeq   A_\alpha \wedge_{P_RX} P_RX \otimes_R M \xrightarrow{\simeq }  \left(A_\alpha \otimes_R M\right).
\end{split}
\end{equation}
Hence we obtain a map, denoted by 
\[s_\alpha: M \wedge \mathrm{TAQ}(A_\alpha/R )  \to \left(A_\alpha \otimes_R M\right).\]
\begin{lemma}
We have that $s_\alpha$ is a splitting of the natural map, denoted by \[D_{A_\alpha}: (A_\alpha \otimes_R M) \to I_M^{A_{\alpha}} \to \left(I_M^{A_\alpha}/(I_M^{A_\alpha})^2\right) \simeq   M\wedge \mathrm{TAQ}(A_\alpha/R).\]
\begin{proof}
The maps making up $D_{A_\alpha}$ respect the filtration of the towers in Lemma \ref{honest} since this map is the projection onto the first layer. The map $s_\alpha$ also respects the filtration of the these towers;  the equivalences of equation (\ref{firstpart}) were proven using $M-\mathrm{THH}$-etale descent which was in turn proven by using the filtration given by this tower.  Hence every map that makes up $s_\alpha$ and $D_{A_\alpha}$ is an equivalence on the first layer.  THus we have that the map $D_{A_\alpha} \circ s_\alpha$ is an equivalence.
\end{proof}
\end{lemma}
Consider the map  \[D_{A}: A \otimes_R M \to M \wedge \mathrm{TAQ}(A/R).\]  We have the equivalence $D_{A} \wedge_A A_\alpha \simeq  D_{A_\alpha}$ since again, etale descent respects the layers of the $\mathrm{TAQ}$-tower functor for the map of $R$-algebras $A \to A_\alpha$.  Thus $s_\alpha$ is a splitting of $D_{A} \wedge_A A_\alpha$. Hence, by the definition of a global covering (Definition \ref{convenient}), there is a (global) splitting of $D_{A}$, denoted by
\[s: M \wedge \mathrm{TAQ}(A/R) \to A \otimes_R M.\]
This $A$-module map induces the $A$-algebra map
\begin{equation}
\label{finally}
P_A(M \wedge \mathrm{TAQ}(A/R)) \to A \otimes_R M.
\end{equation}
The map (\ref{finally}) induces the map 
\[P_{A_\alpha}(M \wedge \mathrm{TAQ}(A_\alpha/R)) \simeq  A_{\alpha} \wedge_A P_A(M \wedge \mathrm{TAQ}) \to A_\alpha \wedge_A A \otimes_R M \simeq  A_\alpha \otimes_R M \]
which is exactly the equivalence of (\ref{firstpart}).  Since global coverings are faithful by Remark \ref{globalfaithful}, we have that (\ref{finally}) is an equivalence.  This finishes the proof of Theorems \ref{mainthm} and \ref{noconnectiveassumption}. \qed

\begin{remark}
\label{tower}
We named the map $D_A$ above as it is the linearization, in the sense of Goodwillie calculus of the functor $- \otimes_R M: R-\mathrm{Alg} \to R-\mathrm{Alg}$ given by $C \mapsto C \otimes_R M$ at the algebra $A$.  Thus what we showed above, is that the Goodwillie tower for $- \otimes_R M$ splits at $A$, when $A/R$ is formally TAQ-smooth.  

\begin{proof}
For a proof that the tower above is actually the Goodwillie tower of the functor $-\otimes_R M$, see \cite[Remark/Claim 2.6]{Min03}.
\end{proof}
\end{remark}

\section{Application to rational higher THH}

We calculate $\mathrm{THH}(KU_\Q/\Q)$ as an application of our HKR theorem, specifically Theorem \ref{noconnectiveassumption}.

\begin{theorem}  We have the following equivalences of $KU_\Q$-algebras which are natural in $M$:
\[\mathrm{THH}(KU_\Q/\Q)\simeq  P_{KU_\Q}{\Sigma^3 H\Q} \]
\[KU_\Q \otimes_{\Q} M\simeq  P_{KU_\Q}(M \wedge \Sigma^2 H\Q).\]

\begin{proof}
Rationally, we have $KU_\Q\simeq  \Q[\beta,\beta^{-1}]$ \,\,\, $|\beta|=2$, is an Eilenberg Maclane spectrum.  We now show that $KU_\Q$ is formally $S^0$-THH-smooth over $\Q$:  

We have the equivalence $\Q[\beta]\simeq  P_{\Q}(\Sigma^2 H\Q)$.  The extension $\Q[\beta] \to \Q[\beta, \beta^{-1}]$ is $S^0$-THH-etale by \ref{localizationexample}. Since we have only one $A_\alpha=A$ in our example, $\Q \to KU_\Q$ is $S^0$-THH smooth.  Hence, we are done by Theorem \ref{noconnectiveassumption}.

\end{proof}

\end{theorem}
\begin{corollary}
The rational $nth$-higher THH and the $nth$ iterated THH of $KU$ is given by 
\[ KU_\Q \otimes_{\Q} S^n \simeq  P_{KU_\Q}(\Sigma^{n+2} H\Q) \] 
\[ KU_\Q \otimes_{\Q} (S^1)^n \simeq  P_{KU_\Q}((S^1)^n \wedge \Sigma^{n+2} H\Q).\]
\end{corollary}

\bibliographystyle{alpha}
\bibliography{bibliography.bib}

\end{document}